\documentclass[11pt]{amsart}

\usepackage{amsmath}
\usepackage{amsfonts}
\usepackage{amsthm}
\usepackage{fixltx2e}

\newtheorem{lemma}{Lemma}
\newtheorem{prop}{Proposition}
\newtheorem{theorem}{Theorem}

\newcommand{\cC}{\mathcal{C}}
\newcommand{\cG}{\mathcal{G}}

\begin{document}
\title[Estimation of CO\textsubscript{2} flux]{Estimation of CO\textsubscript{2} flux from targeted satellite observations: a Bayesian approach}
\author{Graham Cox}
\email{ghcox@email.unc.edu}
\address{Department of Mathematics, UNC Chapel Hill, Phillips Hall CB \#3250, Chapel Hill, NC 27599}

\begin{abstract} We consider the estimation of carbon dioxide flux at the ocean--atmosphere interface, given weighted averages of the mixing ratio in a vertical atmospheric column. In particular we examine the dependence of the posterior covariance on the weighting function used in taking observations, motivated by the fact that this function is instrument-dependent, hence one needs the ability to compare different weights. The estimation problem is considered using a variational data assimilation method, which is shown to admit an equivalent infinite-dimensional Bayesian formulation.

The main tool in our investigation is an explicit formula for the posterior covariance in terms of the prior covariance and observation operator. Using this formula, we compare weighting functions concentrated near the surface of the earth with those concentrated near the top of the atmosphere, in terms of the resulting covariance operators.

We also consider the problem of observational targeting, and ask if it is possible to reduce the covariance in a prescribed direction through an appropriate choice of weighting function. We find that this is not the case---there exist directions in which one can \emph{never} gain information, regardless of the choice of weight.

\smallskip
\noindent \textbf{Keywords.} Carbon flux; Variational data assimilation.
\end{abstract}

\maketitle
%\tableofcontents

\section{Introduction}
It is well known that the effects of trace atmospheric gases, such as CO\textsubscript{2}, are important in long-term climate predictions. In monitoring the concentrations of such atmospheric constituents, it is crucial to understand the fluxes between different components of the climate system (i.e. land, atmosphere, and ocean). These fluxes can be estimated from observations of the concentration (or mixing ratio) of the trace gas in question, coupled with an appropriate transport model (see, for instance, \cite{EM89,NE88}). However, such inversions are strongly dependent on the choice of atmospheric transport model (see \cite{Ge02} and references therein), so there is still disagreement regarding annually-averaged flux magnitudes, to say nothing of variations on seasonal (or even shorter) time scales.

Recent attention has focused on the possibility of obtaining more accurate data in the near future, in the form of satellite-based lidar observations. A number of such observational strategies were recently studied in \cite{He10}. Our objective is to compare these strategies using the methodology of variational data assimilation. The primary technical tool in our investigation is an explicit formula for the posterior covariance, in terms of the (instrument-dependent) observation operator. This formula allows us to compare different observational methods via their respective weighting functions, and discuss the possibility of targeting observations, in terms of the resulting posterior covariances.

Throughout we adopt the guiding principle of \cite{S10} (see also \cite{T05})---that one should avoid discretization of a problem for as long as is reasonably possible, in order to distinguish properties of the original continuous system from nonphysical artifacts of the discretization. We are thus led to model the boundary flux as a (time-dependent) function defined over the entire observational window. As a result, the variational problem (and hence the equivalent Bayesian problem) is infinite-dimensional, with the covariance a bounded, selfadjoint operator. Given a finite amount of data, the prior and posterior covariance operators agree on a subspace of finite codimension. The space on which they differ (i.e. where information has been gained) can be explicitly described in terms of the observational weighting function.

The primary contributions of this paper are the following.

\noindent \textit{1. Bayesian interpretation of the regularized variational problem}

We set up a well-posed variational problem (using Tikhonov regularization) and prove that it possesses a unique minimizer. Moreover, we establish $L^2$ growth estimates on the transport equation and conclude that the regularized problem has a Bayesian interpretation. This is important for two reasons. First, it gives a concrete interpretation of the regularization parameters, and explains how they should be chosen using prior information. Second, it allows us to relate the Hessian of the cost functional to the posterior covariance, and hence interpret our variational computations in terms of the ``decrease in covariance" or ``gain in information" resulting from each observation.

\noindent \textit{2. Explicit computation of posterior covariance operator}

We give an explicit formula for the posterior covariance in terms of the prior covariance and the observational weighting function. In particular, the posterior covariance is related to the Fourier coefficients of the weighting function when it is expanded in the modes of a particular Sturm--Liouville operator. These modes depend on the atmospheric diffusion and transport velocity, and hence establish a clear geometric connection between the underlying atmospheric dynamics and the information gain from each observation. This formula allows us to identify the direction of maximal information gain for a given weighting function---this is useful in comparing different weights, and in targeting directions in which one would like to learn something by choosing an appropriate weight.

\noindent \textit{3. Comparison of high and low altitude observations}

Since the weighting function is instrument-dependent, it is essential that we understand the the experimental ramifications of different instrument designs (say, in terms of the posterior covariance), to ensure that we choose the best weight. Of course the notion of ``best" is very subtle, and depends on what information is considered most relevant.

We compare different weighting functions using the explicit covariance formula described above. It is shown that monotone (as a function of altitude) weights guarantee a certain monotonicity property for the posterior covariance operator. It is also shown that weighting functions concentrated near the surface of the Earth are better at estimating the mean flux over a given time window than weights concentrated near the top of the atmosphere. Since the flux is being estimated right at the ocean--atmosphere interface, it is not particularly surprising that a weight concentrated nearby will produce the best results. However, this is only true of the \emph{mean} flux. The situation becomes much more complicated when one considers the entire covariance operator. In fact it can be shown that, given any two weighting functions, say $\rho_1$ and $\rho_2$, there exists a direction in which $\rho_1$ decreases the covariance more than $\rho_2$.

\noindent \textit{4. Discussion of observational targeting and unobservability}

Our covariance formula is also useful for the problem of observational targeting. Specifically, we ask if, given a fixed direction in the function space that the flux belongs to, one can find an observation operator that decreases the covariance in that direction? In particular we just ask whether or not the covariance can be decreased by some nonzero amount in the given direction (and allow for the possibility that the decrease is greater elsewhere). Somewhat surprisingly, we find that this is not possible---there exist directions along which the prior and posterior covariance operators coincide for \emph{any} choice of observational weighing function. This means there are direction that are completely inaccessible to observations made using the family of observation operators considered in this paper. However, the problem of explicitly describing these unobservable directions, and determining their physical significance (if any), is still open. \\

A related concept in determining the information content of a given observation is the so-called ``influence matrix" \cite{CPA04} or ``hat matrix" \cite{HW78}. This is a diagnostic tool in ordinary or generalized least squares regression which measures the sensitivity of the predicted model output with respect to variations in the data. However, it only measures the change in the \emph{observable quantity} (in our case, a weighted average of the carbon mixing ratio in the atmospheric column). This is not well-suited to the present situation, as we are interested in determining the surface flux, which is a function (hence exists an an infinite-dimensional space), so there is much variability that is not captured by the influence matrix. As noted above, there is no universal notion of an ``ideal" observation operator, and so we work directly with the full posterior covariance operator, in order to determine precisely how knowledge of the flux is improved with each observation.

It well known in the variational data assimilation literature (see e.g. \cite{PVT96,TCBL06,TRC93} and references therein) that there is a strong connection between an observation's spatial location, the generalized Fourier modes of the underlying dynamical system, and the overall sensitivity of the assimilation scheme. For instance, if an observation is made near a spatial maximum of the fastest growing mode, it is possible to match the data by effecting a small change in the initial condition (relative to the assumed prior value). While the connection to the geometry of the modes of the underlying dynamical system is certainly reminiscent of the present work, there is a key difference in that we are interested in estimating the flux (a function of time, measured at a single point in space), rather than the initial condition (a function of space, measured at a single point in time). In fact this distinction is crucial to the phenomenon of unobservable subspaces that we will encounter below.

%%%%%%%%%%%%%%%%%%%%%%%%%%%%%%%%%%%%%%%%
%%%%%%%%%%%%%%%%%%%%%%%%%%%%%%%%%%%%%%%%
%%%%%%%%%%%%%%%%%%%%%%%%%%%%%%%%%%%%%%%%
%%%%%%%%%%%%%%%%%%%%%%%%%%%%%%%%%%%%%%%%
%%%%%%%%%%%%%%%%%%%%%%%%%%%%%%%%%%%%%%%%

\section{The transport model: existence and basic estimates}
\label{sec:model}

We model the atmosphere by a vertical column of finite height $h$. Letting $q(z,t)$ denote the mass fraction of carbon dioxide at height $z$ and time $t$, we assume $q$ evolves by a linear advection--diffusion equation
\begin{align}
	q_t + (wq)_z = (kq_z)_z
	\label{eqn:evol}
\end{align}
where we have used subscripts to denote partial derivatives, and $w$ and $k$ denote the atmospheric velocity and diffusion, respectively, which are allowed to vary with height but are assumed to be constant in time. (See \cite{EM89} and \cite{NE88} for applications of similar linear models to flux estimation problems.) The following assumptions will frequently be needed in the remainder of the paper.
\begin{align*}
	k,w \in C^2[0,h] \tag{A1} \label{ass:reg} \\
	k(z) > 0 \text{ for } z \in [0,h] \tag{A2} \label{ass:ell} \\
	w(0) = w(h) = 0 \tag{A3} \label{bc:noflux}
\end{align*}
The following existence result is standard (cf. Corollary 5.2 of \cite{F64}).

\begin{prop} \label{prop:exist}
Suppose \eqref{ass:reg} and \eqref{ass:ell} are satisfied, and let $T > 0$. For $F \in C[0,T]$ and $q_0 \in C^1[0,h]$, there exists a unique solution to (\ref{eqn:evol}) on $[0,h] \times (0,T]$ that satisfies the boundary conditions
\begin{align}
	q_z(0,t) &= -F(t) \label{bc:flux} \\
	q_z(h,t) & =0 \label{bc:upperflux}
\end{align}
and the initial condition
\begin{align}
	q(z,0) = q_0(z).
	\label{eqn:initial}
\end{align}
\end{prop}

This solution will be denoted $q^F(z,t)$. To understand the physical significance of the function $F(t)$ in the boundary condition \eqref{bc:flux}, we further assume that \eqref{bc:noflux} holds, and compute
\begin{align}
	\frac{d}{dt} \int_0^h q^F(z,t) dz = k(0) F(t)
\label{eqn:totalflux}
\end{align}
for any $t \in (0,T]$. Thus, up to a constant, $F$ denotes the net transport of carbon into or out of the atmosphere. In this computation \eqref{bc:noflux} ensures there is no transport into or out of the atmosphere on account of the advective term $w(z)$. In a 3-dimensional model, the corresponding boundary condition is that the velocity field $\mathbf{w}$ has vanishing normal component at the top and bottom of the atmosphere. Similarly, \eqref{bc:upperflux} ensures there is no diffusive flux at the top of the atmosphere. Our sign convention is chosen so that a positive value of $F$ corresponds to a net flux of carbon into the atmosphere.

We next derive an \textit{a priori} estimate for $q^F$ that will be paramount in establishing the Bayesian formulation in Section \ref{sec:bayes} below. Observe that \eqref{ass:ell} and the continuity of $k$ (which follows from \eqref{ass:reg}) together imply the existence of a number $\epsilon > 0$ such that
\[
	k(z) \geq \epsilon
\]
for all $z \in [0,h]$.

\begin{prop} \label{prop:energy}
Suppose \eqref{ass:reg}, \eqref{ass:ell} and \eqref{bc:noflux} are satisfied. There exists a positive constant $K$, depending on $\sup |w|$, $k(0)$, $\epsilon$ and $h$, such that if $T > 0$ and $F \in C[0,T]$, then
\begin{align}
	\|q^F(\cdot,t)\|^2_{L^2(0,h)} \leq K e^{Kt} \left[ \left(1+t \right) \|q_0\|^2_{L^2(0,h)}  + \left(1 + t^2 \right) \|F\|^2_{L^2(0,t)}  \right]
\end{align}
for any $t \in [0, T]$, where $q^F$ denotes the unique solution to \eqref{eqn:evol}--\eqref{eqn:initial} given by Proposition \ref{prop:exist}.
\end{prop}
%
%If $C$ is additionally allowed to depend on the maximal time $T$, the $t$ and $t^{2/3}$ terms can be eliminated.

The estimate is proved by an application of Gronwall's inequality. Along the way it will be necessary to show that $\|q^F_z\|_{L^2(0,h)}$ controls the pointwise values of $q^F$. To do so, we first establish that there is at least one point in space where the solution is not too large. To simplify notation we abbreviate $q^F$ to $q$ for the remainder of the section.

\begin{lemma} For each $t_* \geq 0$ there exists $z_* \in [0,h]$ such that
\begin{align*}
	q(z_*,t_*) = \frac{1}{h} \left( \int_0^h q_0(z) dz + k(0) \int_0^{t_*} F(t) dt \right).
\end{align*}
\label{lemma:qmin}
\end{lemma}

\begin{proof} Integrating \eqref{eqn:totalflux} with respect to $t$ we find
\begin{align*}
	\int_0^h q(z,t_*) dz = \int_0^h q_0(z) dz + k(0) \int_0^{t_*} F(t) dt
\end{align*}
and the result follows from the mean value theorem for integrals.
\end{proof}

The Cauchy--Schwarz inequality implies
\begin{align}
	|q(z_*,t_*)| \leq \frac{1}{h} \left( \sqrt{h} \|q_0\|_{L^2(0,h)} + k(0) \sqrt{t_*} \|F\|_{L^2(0,t_*)} \right),
	\label{eqn:qmin}
\end{align}
so we have found a point in space where the value of $q$ is controlled by the flux and the initial data. This allows us to estimate $q$ at any other point.

\begin{lemma} For any $(z,t) \in [0,h] \times (0,T]$ we have
%\begin{align}
%	|q^F(z,t)| \leq \frac{1}{h} \left( \sqrt{h} \|q_0\|_{L^2(0,h)} + k(0) \sqrt{t} \|F\|_{L^2(0,t)}  \right) + \sqrt{h} \|q_z(\cdot,t) \|_{L^2(0,h)}.
%\end{align}
%or
\begin{align}
	|q(z,t)|^2 \leq A \left( \|q_0\|_{L^2(0,h)}^2 + \|q_z(\cdot,t) \|_{L^2(0,h)}^2 + t \|F\|^2_{L^2(0,t)} \right), \label{eqn:qbound}
\end{align}
where $A > 0$ depends only on $h$ and $k(0)$.
\end{lemma}

\begin{proof} Let $t_* \in [0,T]$, and choose $z_*$ according to Lemma \ref{lemma:qmin}. Then for any $z \in [0,h]$ we can write
\begin{align*}
	q(z,t_*) = q(z_*,t_*) + \int_{z_*}^z q_z(\zeta, t_*) d\zeta.
\end{align*}
We observe that
\begin{align*}
	\int_{z_*}^z q_z(\zeta, t_*) d\zeta \leq \sqrt{h} \|q_z(\cdot,t_*) \|_{L^2(0,h)},
\end{align*}
with the right-hand side independent of $z_*$. Since $t_* \in (0,T]$ was arbitrary, the result follows from \eqref{eqn:qmin}.
\end{proof}

%It will be convenient to write this in the form
%\begin{align}
%	|q(z,t)|^2 \leq A \left( \|q_0\|^2 + \|q_z(t) \|^2 + t \int_0^t |F(s)|^2 ds \right),
%\end{align}
%where $A$ depends only on $h$.
We are now ready to prove the main energy estimate.

\begin{proof}[Proof of Proposition \ref{prop:energy}] Differentiating and then integrating by parts, we find that
\begin{align*}
    \frac{1}{2} \frac{d}{dt} \int_0^h q(z,t)^2 dz = k(0,t) q(0,t) F(t) - \int_0^h k q_z^2 dz + \int_0^h w q q_z dz.
\end{align*}
For any positive constant $\alpha$, the inequality of arithmetic and geometric means implies
\begin{align*}
	q(0,t) F(t) &\leq \frac{\alpha}{2} |q(0,t)|^2 + \frac{1}{2 \alpha} |F(t)|^2 \\
	& \leq \frac{\alpha A}{2} \left( \|q_0\|_{L^2(0,h)}^2 + \|q_z(\cdot,t) \|_{L^2(0,h)}^2 + t \|F\|^2_{L^2(0,t)} \right) + \frac{1}{2\alpha} |F(t)|^2,
\end{align*}
where in the second inequality we have used \eqref{eqn:qbound} to estimate $|q(0,t)|$.
Similarly,
\begin{align*}
	\int_0^h w q q_z dz \leq \frac{\sup |w|}{2} \left( \beta \|q_z\|^2 + \beta^{-1} \|q\|^2 \right)
\end{align*}
for any $\beta > 0$. Choosing $\alpha$ and $\beta$ small enough that
\begin{align*}
	\frac{\alpha A k(0)}{2} + \frac{\beta \sup |w|}{2} \leq \epsilon
\end{align*}
we arrive at the inequality
\begin{align*}
    \frac{d}{dt} \|q(\cdot,t)\|_{L^2(0,h)}^2 &\leq K \left[ \|q_0\|_{L^2(0,h)}^2 + \|q(\cdot,t)\|_{L^2(0,h)}^2 + |F(t)|^2 + t \|F\|^2_{L^2(0,t)} \right],
\end{align*}
where $K$ depends on $\sup|w|$, $k(0)$, $\epsilon$ and $h$. The result then follows from the differential form of Gronwall's inequality (see, for instance, \cite{E10}).
\end{proof}

%%%%%%%%%%%%%%%%%%%%%%%%%%%%%%%%%%%%%%%%
%%%%%%%%%%%%%%%%%%%%%%%%%%%%%%%%%%%%%%%%
%%%%%%%%%%%%%%%%%%%%%%%%%%%%%%%%%%%%%%%%
%%%%%%%%%%%%%%%%%%%%%%%%%%%%%%%%%%%%%%%%
%%%%%%%%%%%%%%%%%%%%%%%%%%%%%%%%%%%%%%%%

\section{Lidar observations and the variational problem}

For simplicity we restrict our attention to satellite-based lidar (light radar) measurements of the carbon mixing ratio---see \cite{K04} for a physical implementation and preliminary sensitivity analysis of a DIAL (differential lidar) system. In this scenario each observation is given by a weighted average of the mixing ratio over the atmospheric column (cf. Equation (4) in \cite{He10}), with respect to a \emph{weighting function} $\rho(z)$ that depends on the observational instrument as well as background atmospheric factors, such as temperature and surface albedo.

Given such a weighting function, we define the corresponding \emph{observation operator} $H: L^2(0,h) \rightarrow \mathbb{R}$ by
\begin{align}
	H q := \int_0^h \rho(z) q(z) dz.
	\label{eqn:obs}
\end{align}
%The kernel $\rho$ is instrument dependent, and additionally depends on certain background atmospheric factors, such as temperature and surface albedo. (See \cite{He10} for further details and examples of weights for specific instruments.)
Since the weighting function is instrument dependent we will think of it as a variable, and study the extent to which one can maximize the information gained from each observation through judicious choice of $\rho$. This approach is justified by the strong qualitative dependence of $\rho$ on the observational frequency (cf. Figure 1 in \cite{E08} and Figure 3 in \cite{He10}).

We assume throughout that $\rho$ is a nonnegative, square-integrable function, hence $H$ is a bounded operator, with adjoint $H^*: \mathbb{R} \rightarrow L^2(0,h)$ satisfying
\[
	\left<H f, a \right>_{\mathbb{R}} = \left<f, H^*a \right>_{L^2(0,h)}
\]
for each $f \in L^2(0,h)$ and $a \in \mathbb{R}$. It follows easily that $H^*a = a\rho$ for $a \in \mathbb{R}$.

Given such weighted observations $y_1, \ldots, y_N$ of the mixing ratio at times $t_1 < \cdots < t_N$, we need to estimate the boundary flux $F(t)$ that best matches the observations, in a sense to be made precise below. We assume that the $i$th observation is subject to normally distributed noise, with variance $r_i^2$ , and recall that $q^F$ denotes the solution to the carbon transport equation \eqref{eqn:evol} with prescribed boundary flux $F$ (which necessarily exists by Proposition \ref{prop:exist}). The problem of finding $F$ that minimizes
\begin{align*}
	\frac{1}{2} \sum_{i=1}^N r_i^{-2} \left| H q^F(\cdot, t_i) - y_i \right|^2,
\end{align*}
is ill-posed, so it is necessary to regularize the problem in order to obtain a meaningful answer. We follow the approach of Tikhonov (see e.g. \cite{EH96,F74,L71,T63}) and define a regularized cost functional
\begin{align}
	J(F) = \frac{1}{2} \sum_{i=1}^N r_i^{-2} \left| H q^F(\cdot, t_i) - y_i \right|^2 + \frac{1}{2} \|\mathcal{C}_0^{-1/2}(F - F_0) \|_{L^2(0,t_N)}^2,
	\label{cost}
\end{align}
for some given function $F_0 \in L^2(0,t_N)$ and selfadjoint, positive semidefinite, trace class operator $\cC_0$ on $L^2(0,t_N)$. We define the associated \emph{Cameron--Martin space} by $E := \text{Im}(\cC_0^{1/2})$.

Note that the regularized cost functional in \eqref{cost} is formally the same as the cost function used in 4D-Var data assimilation (see \cite{DT86,L86}, and \cite{S10} for the infinite-dimensional case). However, we will refrain from calling our assimilation scheme ``4D-Var" because it is in fact only two-dimensional (counting both spatial and temporal variables) on account of the simplified physical model under consideration.

An important example is $\cC_0 = - \sigma^2 \Delta^{-1}$, where $\Delta = (d/dt)^2$ with domain $H^2(0,t_N) \cap H^1_0(0,t_N)$, and $\sigma > 0$. In this case $E = H^1_0(0,t_N)$ and the cost functional is
\begin{align*}
	J(F) = \frac{1}{2} \sum_{i=1}^N r_i^{-2} \left| H q^F(\cdot, t_i) - y_i \right|^2 + \frac{1}{2\sigma^2} \|F - F_0 \|_{H^1_0(0,t_N)}^2.
\end{align*}

There is a standard existence result for the regularized problem.

\begin{prop} \label{prop:Jmin}
If $E$ is compactly embedded in $L^2(0,t_N)$, then there exists a unique function $\bar{F} \in F_0 + E$ such that
\[
	J(\bar{F}) = \inf_{F \in F_0 + E} J(F).
\]
\end{prop}
%While the result itself is elementary, we mention it to emphasize the similarities (observed in \cite{S10}) between the hypotheses of Proposition \ref{prop:Jmin} and those of Proposition \ref{prop:bayes} below.

To prove this we consider the map $\cG: L^2(0,t_N) \rightarrow \mathbb{R}^N$ defined by
\begin{align} \label{eqn:Gdef}
	\mathcal{G}(F) := \left( H q^F(\cdot, t_1), \ldots, H q^F(\cdot, t_N) \right).
\end{align}
Appealing to Proposition \ref{prop:exist}, $\cG$ is defined for all $F \in C[0,t_N]$. Using Proposition \ref{prop:energy} and the fact that $C[0,t_N] \subset L^2(0,t_N)$ is dense, we can extend $\cG$ uniquely to all of $L^2(0,t_N)$.

\begin{proof}[Proof of Proposition \ref{prop:Jmin}] The result follows from Theorem 5.4 in \cite{S10} once we show that the function $\cG$ defined in \eqref{eqn:Gdef} satisfies the following conditions:
\begin{enumerate}
	\item[(i)] for every $\epsilon > 0$ there is an $M = M(\epsilon) \in \mathbb{R}$ such that, for all $F \in L^2(0,t_N)$,
	\begin{align*}
		| \mathcal{G}(F) | \leq \exp \left( \epsilon \|F\|_{L^2(0,t_N)}^2 + M \right);
	\end{align*}
	\item[(ii)] for every $r > 0$ there is a $K = K(r)$ such that, for all $F_1, F_2$ in the closed ball $B_r(0) \subset L^2(0,t_N)$,
	\begin{align*}
		| \mathcal{G}(F_1) - \mathcal{G}(F_2) | \leq K \|F_1 - F_2\|_{L^2(0,t_N)}.
	\end{align*}
\end{enumerate}
Both conditions are immediate consequences of the estimate given in Proposition \ref{prop:energy}.
\end{proof}

%%%%%%%%%%%%%%%%%%%%%%%%%%%%%%%%%%%%%%%%
%%%%%%%%%%%%%%%%%%%%%%%%%%%%%%%%%%%%%%%%
%%%%%%%%%%%%%%%%%%%%%%%%%%%%%%%%%%%%%%%%
%%%%%%%%%%%%%%%%%%%%%%%%%%%%%%%%%%%%%%%%
%%%%%%%%%%%%%%%%%%%%%%%%%%%%%%%%%%%%%%%%

\section{The Bayesian interpretation}
\label{sec:bayes}

Proposition \ref{prop:Jmin} guarantees the existence of a unique minimizer for the regularized variational problem. In this section we observe that the regularized problem has a natural Bayesian interpretation in which the minimum of $J$ corresponds to the MAP (maximum \textit{a posteriori}) estimator. In particular, the regularization term in \eqref{cost} corresponds to a Gaussian prior distribution with mean $F_0$ and covariance $\cC_0$. This interpretation makes clear the role played by $F_0$ and $\cC_0$, and can be used to guide our choice of these quantities (based on prior knowledge) to ensure the regularized problem has a meaningful solution.

Since the evolution equation and the observation operator (together represented by the function $\cG$ above) are both linear, if we assume the observational noise is normally distributed, it follows that the family of Gaussian prior distributions is conjugate, in the sense that they necessarily yield Gaussian posteriors. The linear estimation problem in Hilbert spaces is well-studied (see \cite{F70,M84} and references contained therein, and \cite{F91} for a Bayesian perspective). The reader is also advised to consult \cite{S10} for a modern, unified presentation, containing many examples of both linear and nonlinear problems.

The assumption that $F_0 \in L^2(0,t_N)$ and $\cC_0$ is selfadjoint, positive semidefinite and trace class guarantees the existence of a Gaussian measure $\mu_0$ with mean and covariance $F_0$ and $\cC_0$, respectively \cite{R71}. This will be the prior distribution in our Bayesian formulation.

Given the estimates in Proposition \ref{prop:energy} (cf. the proof of Proposition \ref{prop:Jmin} above), the following result is an immediate consequence of Corollary 4.4 in \cite{S10}.

\begin{prop} \label{prop:bayes}
Let $\mu_0$ be a Gaussian measure such that $\mu_0 \left[ L^2(0,t_N) \right] = 1$, and suppose that \eqref{ass:reg}, \eqref{ass:ell} and \eqref{bc:noflux} are satisfied. Given a set of observations $\{(y_i,t_i)\}_{i=1}^N$, where the $i$th observation has normally distributed noise with covariance $r_i^2$, there is a well-defined posterior measure $\mu_y$ on $L^2(0,t_N)$, with Radon--Nikodym derivative
\begin{align}
	\frac{d \mu_y}{d\mu_0}(F) \propto \exp \left\{ - \sum_{i=1}^N r_i^{-2} \left| H q^F(\cdot, t_i) - y_i \right|^2 \right\}.
\end{align}
Moreover, the posterior mean and covariance are continuous with respect to the data $y = (y_1, \ldots, y_N) \in \mathbb{R}^N$.
\end{prop}

The condition $\mu_0 \left[ L^2(0,t_N) \right] = 1$ ensures that the Cameron--Martin space $E = \text{Im}(\cC_o^{1/2})$ is compactly embedded in $L^2(0,t_N)$, so the existence result of Proposition \ref{prop:Jmin} applies.

%
%
%
%
%Following the strategy of \cite{S10}, we need the following in order to have a well-defined Bayesian formulation:
%\begin{enumerate}
%	\item a Banach space $X$ of fluxes for which the forward problem is well-posed, and satisfies a certain \emph{a priori} estimate;
%	\item a prior measure $\mu_0$ such that $\mu_0[X] = 1$.
%\end{enumerate}
%
%The first condition ensures uniform control of solutions to the forward equation (and hence the resulting observations); the second ensures that these estimates hold for almost every flux drawn from the prior distribution.

A standard choice for the prior is a Gaussian distribution with covariance proportional to the inverse Laplacian $-\Delta^{-1}$. To make sense of this we must impose boundary conditions to ensure that the Laplacian is invertible. Two common choices are:
\begin{enumerate}
	\item Dirichlet boundary conditions, i.e. $\Delta: H^2 \cap H^1_0 \rightarrow L^2$;
	\item periodic boundary conditions, i.e. $\Delta: \mathring{H}^2_{\text{per}} \rightarrow L^2$, where $\mathring{H}^2_{per}$ is the space of periodic functions in $H^2$ with zero mean.
\end{enumerate}
The corresponding Cameron--Martin spaces are $H^1_0$ and $\mathring{H}^1_{\text{per}}$, respectively. In the first example we are imposing the prior assumption that the flux agrees with the prior mean $F_0$ at the endpoints $\{0,t_N\}$, as can easily be seen from the Euler--Lagrange equation for (\ref{cost}). On the other hand, the periodic, zero mean assumption is only justified if all long-term trends have been subtracted from the data, so it would not be appropriate for determining the average flux over a given window of time. 

However, it is important to note that our main formula (Theorem \ref{thm:cov} below) expresses the \emph{difference} between the prior and posterior covariance operators, so for our purposes the particular choice of $\mu_0$ is not important as long as the conditions of Proposition \ref{prop:bayes} are satisfied.

%%%%%%%%%%%%%%%%%%%%%%%%%%%%%%%%%%%%%%%%
%%%%%%%%%%%%%%%%%%%%%%%%%%%%%%%%%%%%%%%%
%%%%%%%%%%%%%%%%%%%%%%%%%%%%%%%%%%%%%%%%
%%%%%%%%%%%%%%%%%%%%%%%%%%%%%%%%%%%%%%%%
%%%%%%%%%%%%%%%%%%%%%%%%%%%%%%%%%%%%%%%%

\section{The posterior covariance}
\label{sec:cov}

We denote the posterior covariance operator by $\cC_1$. The main result of this section is an explicit formula for $\cC_1$ in terms of $\cC_0$ and the weighting function $\rho$ in the observation operator defined by \eqref{eqn:obs}.

Letting $Lq = (kq_z)_z - (wq)_z$ (so that the evolution equation \eqref{eqn:evol} is $q_t = Lq$), we define the adjoint operator $L^*$ by $L^*p =  (kp_z)_z + w p_z$. Letting $' = \frac{d}{dz}$, and assuming that $p'(h) = q'(h) = 0$ and (\ref{bc:noflux}) holds, we find that
\begin{align}
	\int_0^h \left[p(Lq) - (L^*p)q \right]dz = k(0) \left[ q(0) p'(0) - p(0) q'(0) \right] 
	\label{eqn:ibp}
\end{align}
for all $p,q$ in the Sobolev space $H^2(0,h)$ (see \cite{E10}).

It is well known (see, for instance, \cite{DS63}) that the Sturm--Liouville problem
%\begin{align}
%	L^* p + \lambda p = 0 \label{eqn:SL} \\
%	p'(0) = p'(h) = 0 
%\end{align}
\begin{align}
	L^* p + \lambda p = 0, \ \ p'(0) = p'(h) = 0  \label{eqn:SL}
\end{align}
has a sequence of simple eigenvalues $\lambda_0 < \lambda_1 < \cdots$, with corresponding eigenfunctions $\{\rho_n\}$ forming a basis for $L^2(0,h)$. Since $\rho_n'(0) = 0$ but $\rho_n$ is not identically zero, standard uniqueness results imply that $\rho_n(0) \neq 0$. Therefore we can assume that the eigenfunctions have been normalized to have $\rho_n(0) = 1$ for each $n$. It is easy to verify that
\begin{align} \label{eqn:adjoint}
	\int_0^h \rho_m (L^*\rho_n) \mu(z) dz = - \int_0^h k \rho_{m}' \rho_{n}' \mu(z) dz
\end{align}
for all $m,n \geq 0$, where we have set $\mu(z) = \exp \int w(z) / k(z) dz$, hence
\[
	\int_0^h \rho_m(z) \rho_n(z) \mu(z) dz = 0
\]
for $m \neq n$.

The normalization $\rho_n(0) = 1$ is chosen to simplify the computation of the posterior covariance. An undesirable consequence of this choice is that the sequence $\{\rho_n\}$ is not necessarily orthonormal (with respect to either $dz$ or $\mu(z) dz$), so we must exhibit some caution when dealing with series of the form $\sum a_n \rho_n$.

\begin{lemma} \label{lem:summability}
There is a constant $C > 0$ such that
\[
	C^{-1} \leq \int_0^h \rho_n(z)^2\mu(z) dz \leq C
\]
for all $n \geq 0$. Therefore, the series
\[
	\sum_{n=0}^{\infty} a_n \rho_n
\]
converges to a function in $L^2(0,h)$ if and only if $\{a_n\} \in \ell^2$.
\end{lemma}

\begin{proof} We write $\rho_n(z) = \tilde{\rho}_n(z) / \tilde{\rho}_n(0)$, where $\{\tilde{\rho}_n\}$ are the corresponding eigenfunctions normalized to have
\[
	\int_0^h \tilde{\rho}_n(z)^2 \mu(z) dz = 1.
\]
The result follows from uniform upper and lower bounds on $|\tilde{\rho}_n(0)|$, as can be found in Chapter V.11.5 of \cite{CH53}.
\end{proof}

It is easy to see that $\lambda_0 = 0$ (with $\rho_0 \equiv 1$), hence all of the eigenvalues are nonnegative. Moreover, we have the asymptotic formula (see \cite{CH53})
\begin{align}
	\lambda_n = cn^2 + \mathcal{O}(1)
	\label{eqn:growth}
\end{align}
for some positive constant $c$. This growth estimate will be of key importance in Section \ref{sec:targeting}.

We are now ready to compute the posterior covariance operator. We write the observational weighting function as
\begin{align}
	\rho = \sum_{n=0}^{\infty} a_n \rho_n,
	\label{eqn:kernel}
\end{align}
with $\{\rho_n\}$ as defined above, then define
\begin{align}
	G_i(t) := k(0) r_i^{-1} \left( \sum_{n=0}^{\infty} a_n e^{\lambda_n(t-t_i)} \right) \chi_{[0,t_i]}(t)
	\label{eqn:basis}
\end{align}
for $1 \leq i \leq N$, where $\{t_i\}$ are the observation times and $\chi_{[0,t_i]}$ denotes the characteristic function of the interval $[0,t_i]$.

\begin{lemma} The functions $\{G_i\}$ defined in \eqref{eqn:basis} satisfy $G_i \in L^2(0,t_N)$ for $1 \leq i \leq N$.
\end{lemma}

\begin{proof} Since $\rho \in L^2(0,h)$, it follows from Lemma \ref{lem:summability} that $\{a_n\} \in \ell^2$. Using \eqref{eqn:growth} we find that $\{e^{\lambda_n(t-t_i)}\}_{n=0}^{\infty} \in \ell^2$ for $t < t_i$, with
\[
	\left\| e^{\lambda_n(t-t_i)} \right\|^2_{\ell^2} \leq \frac{A}{\sqrt{t_i-t}}
\]
for some $A > 0$ independent of $n$ and $t$. The Cauchy--Schwarz inequality implies
\[
	|G_i(t)|^2 \leq \frac{A'}{\sqrt{t_i-t}}
\]
for some $A' > 0$, and the result follows.
\end{proof}

We denote by $D(\mathcal{C}_0) \subset L^2(0,t_N)$ the domain of the prior covariance operator. The main result of this section is the following.

\begin{theorem} \label{thm:cov}
Suppose \eqref{ass:reg}, \eqref{ass:ell} and \eqref{bc:noflux} are satisfied. Then the posterior covariance is given by
\begin{align}
	\mathcal{C}_1^{-1} G &= \mathcal{C}_0^{-1} G + \sum_{i=1}^N \left( \int_0^{t_N} G(s) G_i(s) ds \right) G_i
	\label{eqn:cov}
\end{align}
for all $G \in D(\mathcal{C}_0)$.
\end{theorem}

We will often find it convenient to use an integral version of this formula,
\begin{align}
	\| \mathcal{C}_1^{-1/2} G \|_{L^2(0,t_N)}^2 = \| \mathcal{C}_0^{-1/2} G \|_{L^2(0,t_N)}^2 + \sum_{i=1}^N \left( \int_0^{t_N} G(t) G_i(t) dt \right)^2,
	\label{eqn:covL2}
\end{align}
which shows that the information gain from the $i$th observation is maximized in the direction spanned by $G_i$, and the prior and posterior covariance operators agree on the codimension one subspace orthogonal to $G_i$. While it is \textit{a priori} clear that a single scalar observation will only yield new information in one dimension, the importance of \eqref{eqn:covL2} is that it allows us to compute this direction explicitly in terms of the weight $\rho$. Applications of this formula to the carbon flux problem are given in Sections \ref{sec:altitude} and \ref{sec:targeting}.

%The above formula also makes explicit the intuitively clear fact that an observation at time $t_i$ only carries information about the flux for times $t \leq t_i$.

We start the proof by computing the second variation of the cost functional \eqref{cost}, and relating it to the posterior covariance operator. We recall that $E = \text{Im}(\cC_0^{1/2})$ denotes the Cameron--Martin space associated to the prior covariance operator, which is compactly embedded in $L^2(0,t_N)$.

\begin{lemma} Let $G \in E$. Then
\begin{align}
	\left\| \mathcal{C}_1^{-1/2} G \right\|_{L^2(0,t_N)} = \sum_{i=1}^N r_i^{-2} \left| H \mathring{q}^G(\cdot, t_i) \right|^2 + \left\| \mathcal{C}_0^{-1/2} G \right\|_{L^2(0,t_N)},
		\label{cov}
\end{align}
where $\mathring{q}^G$ denotes the unique solution to (\ref{eqn:evol}) with boundary conditions
\[
	\mathring{q}^G_z(0,t) = -G(t), \ \ \mathring{q}^G_z(h,t) =0
\]
and initial condition
\[
	\mathring{q}^G(z,0) = 0.
\]
\end{lemma}

\begin{proof}
Viewing the solution to the transport equation given in Proposition \ref{prop:exist} as a map $F \mapsto q^F$, the linearization is given by
\begin{align*}
	\left. \frac{d}{ds} \right|_{s=0} q^{F+sG}
	= \mathring{q}^G,
\end{align*}
%(Recall that the initial condition $q_0$ is fixed throughout.) Computing directly, we then find
so we obtain
\begin{align*}
	D^2J(F)(G,G) =& \left. \frac{d^2}{ds^2} \right|_{s=0} J(F+sG) \\
	=& \sum_{i=1}^N r_i^{-2} \left| H \mathring{q}^G(\cdot, t_i) \right|^2 + \left\| \mathcal{C}_0^{-1/2} G \right\|_{L^2(0,t_N)}
\end{align*}
for the Hessian of $J$.

Since the forward problem is linear, the cost functional is quadratic and has a unique minimizer, which we denote by $\bar{F}$. Then Taylor's theorem implies
\begin{align*}
	J(\bar{F}+sG) = J(\bar{F}) + \frac{s^2}{2} D^2 J(\bar{F})(G,G)
\end{align*}
for any $G$. We also have that the posterior measure is Gaussian (cf. Theorem 6.20 and Example 6.23 in \cite{S10}, and also Section 5 of \cite{ALS13}) hence
\begin{align*}
	J(F) = J(\bar{F}) + \frac{1}{2} \left< \mathcal{C}_1^{-1/2} (F - \bar{F}),  \mathcal{C}_1^{-1/2} (F - \bar{F}) \right> _{L^2(0,T)}.
\end{align*}
Setting $F = \bar{F} + sG$ and equating the above expressions, we find
\begin{align}
	\left\| \mathcal{C}_1^{-1/2} G \right\| _{L^2(0,t_N)} = D^2J(\bar{F})(G,G),
\end{align}
which completes the proof.
\end{proof}

We use the spectral decomposition of the observational kernel, as prescribed in (\ref{eqn:kernel}), to complete the proof of Theorem \ref{thm:cov}.

\begin{proof}[Proof of Theorem \ref{thm:cov}] For the $i$th observation we have
\begin{align*}
	H \mathring{q}^G(\cdot, t_i) = \sum_{n=0}^{\infty} a_n \int_0^h \rho_n(z) \mathring{q}^G(z,t_i) dz.
\end{align*}
We focus on the $n$th summand, which we write as $a_n H_n(t_i)$ with $H_n$ defined by
\begin{align*}
	H_n(t) = \int_0^h \rho_n(z) \mathring{q}^G(z,t) dz.
\end{align*}
Differentiating and then applying (\ref{eqn:ibp}), we have
\begin{align*}
	\frac{d}{dt} H_n(t) &= \int_0^h \rho_n(z) \frac{\partial \mathring{q}^G(z,t)}{\partial t} dz \\
	&= k(0) G(t) - \lambda_n H_n(t),
\end{align*}
because $\mathring{q}^G_t = L\mathring{q}^G$, $L^* \rho_n = \lambda_n \rho_n$, and $\rho_n(0) = 1$. Using an integrating factor, we find
\begin{align*}
	H_n(t) = k(0) e^{-\lambda_n t} \int_0^t G(s) e^{\lambda_n s} ds.
\end{align*}
Summing over $n$, it follows that
\begin{align*}
	H\mathring{q}^G(\cdot, t_i) = k(0) \sum_{n=0}^{\infty} a_n e^{-\lambda_n t_i} \int_0^{t_i} G(s) e^{\lambda_n s} ds
\end{align*}
for each $i$. Squaring both sides and summing over $i$, we obtain (\ref{eqn:covL2}). Finally, (\ref{eqn:cov}) follows from (\ref{eqn:covL2}) via a standard polarization argument.
\end{proof}

%%%%%%%%%%%%%%%%%%%%%%%%%%%%%%%%%%%%%%%%
%%%%%%%%%%%%%%%%%%%%%%%%%%%%%%%%%%%%%%%%
%%%%%%%%%%%%%%%%%%%%%%%%%%%%%%%%%%%%%%%%
%%%%%%%%%%%%%%%%%%%%%%%%%%%%%%%%%%%%%%%%
%%%%%%%%%%%%%%%%%%%%%%%%%%%%%%%%%%%%%%%%

\section{High vs. low altitude observations}
\label{sec:altitude}
Our first application of Theorem \ref{thm:cov} is to compare observational weighting functions concentrated near the surface of the Earth with weights concentrated near the top of the atmosphere. We start with a simple example, letting
\begin{align}
	\rho_{\pm}(z) = \rho_0(z) \pm \rho_1(z),
\end{align}
where $\rho_0$ and $\rho_1$ are the first two adjoint eigenfunctions defined in (\ref{eqn:SL}) with the normalization $\rho_0(0) = \rho_1(0) = 1.$ In the case that $w=0$ and $k=1$, this is just
\begin{align*}
	\rho_{\pm}(z) = 1 \pm \cos (\pi z /h),
\end{align*}
so we will think of $\rho_+$ and $\rho_-$ as being concentrated near the surface of the earth and the top of the atmosphere, respectively.

We assume a single observation, at time $T>0$. According to \eqref{eqn:covL2}, information is gained in the directions spanned by
\begin{align*}
	G_{\pm}(t) = 1 \pm e^{\lambda_1(t-T)}.
\end{align*}
Note that $G_+$ and $G_-$ are strictly increasing and decreasing, respectively. Thus an observation made by $\rho_+$ will yield a higher gain in information than $\rho_-$ on the subspace of fluxes that are increasing in magnitude (with respect to time). It turns out this monotonicity property holds more generally.

\begin{theorem} \label{thm:monotone}
Let $\rho \in C^1[0,h]$ be increasing (decreasing) with respect to $z$. For each $1 \leq i \leq N$, the function $G_i \in L^2(0,t_N)$ defined in (\ref{eqn:basis}) is decreasing (increasing) with respect to $t$.
\end{theorem}

\begin{proof} It is clear that the function
\begin{align*}
	p_i(z,t) := \sum_{n=0}^{\infty} a_n e^{\lambda_n(t-t_i)} \rho_n(z)
\end{align*}
satisfies the adjoint equation $p_{it} + L^* p_i = 0$ on $[0,h] \times (-\infty,t_i)$, with Neumann boundary conditions. Therefore
\begin{align*}
	s(z,t) := p_{iz}(z,t_i-t)
\end{align*}
satisfies the equation
\begin{align*}
	s_t = (ks)_{zz} + (ws)_z
\end{align*}
on $[0,h] \times (0,\infty)$, with Dirichlet boundary conditions. By construction we have $s(z,0) \geq 0$ and $s(0,t) = s(h,t) = 0$,  so the maximum principle (e.g. Lemma 5 of \cite{F64}) implies $s(z,t) \geq 0$ for all $(z,t) \in [0,h] \times [0,\infty)$. It follows that $s_z(0,t) \geq 0$ for $t \geq 0$.

We thus have $p_{iz}(0,t) = 0$ and $p_{izz}(0,t) \geq 0$ for all $t \leq t_i$. Evaluating the adjoint equation at $z=0$, we obtain
\begin{align*}
	p_{it}(0,t) = -k(0) p_{izz}(0,t) \leq 0.
\end{align*}
Recalling the normalization $\rho_n(0) = 1$, we see that $p_i(0,t) \propto G_i(t)$, with a positive constant of proportionality, and the result follows.
\end{proof}

While the above result demonstrates a strong connection between the geometry of $\rho(z)$ and the resulting posterior covariance (via $G_i$), many questions of course remain. For instance, while it was shown that $\rho_+$ ($\rho_-$) gives more information on the subspaces of fluxes that are increasing (decreasing) with respect to time, it is not immediately clear which is better suited to determining the \emph{mean} flux over the observational time window.

From \eqref{eqn:covL2} we see that determining which of $\rho_{\pm}$ yields more information about the mean flux on $[0,T]$ is equivalent to determining which of the functions $G_{\pm}$ has a larger (in absolute value) orthogonal projection onto the space of constant functions. We integrate to find
\begin{align*}
	\int_0^T G_{\pm}(t) dt = T \pm \frac{1 - e^{-\lambda_1 T}}{\lambda_1}.
\end{align*}
A simple convexity estimate yields $1 - e^{-\lambda_1 T} \leq \lambda_1 T$. This implies both of the integrals above are positive, and so
\begin{align*}
	\left| \int_0^T G_+(t) dt \right| - \left| \int_0^T G_-(t) dt \right| &= 2\frac{1 - e^{-\lambda_1 T}}{\lambda_1},
\end{align*}
which is clearly positive.

This example suggests that for the purpose of determining the \emph{mean} flux, it is advisable to choose a weighting function that is concentrated near the surface of the Earth. It also illustrates that the comparison of weights is a rather subtle business, as there is no universal notion of a ``best" observation operator.

%%%%%%%%%%%%%%%%%%%%%%%%%%%%%%%%%%%%%%%%
%%%%%%%%%%%%%%%%%%%%%%%%%%%%%%%%%%%%%%%%
%%%%%%%%%%%%%%%%%%%%%%%%%%%%%%%%%%%%%%%%
%%%%%%%%%%%%%%%%%%%%%%%%%%%%%%%%%%%%%%%%
%%%%%%%%%%%%%%%%%%%%%%%%%%%%%%%%%%%%%%%%

\section{Targeting and unobservable dimensions}
\label{sec:targeting}
Our second application is to observational targeting: given a ``direction of interest" $\widehat{G}$ and a time $t_i$, does there exist a weighting function $\rho$ such that the corresponding $G_i$ is proportional to $\widehat{G}$? In other words, can one construct an observation operator of the form \eqref{eqn:obs} that yields maximal information gain in a given direction $\widehat{G}$?

According to Theorem \ref{thm:cov}, the existence of a weight $\rho$ with the desired properties is equivalent to the existence of a sequence $\{a_n\} \in \ell^2$ such that
\[
	\widehat{G}(t) = \sum_{n=0}^{\infty} a_n e^{\lambda_n(t-t_i)}
\]
in $L^2(0,t_i)$. There are several nontrivial technical details to be resolved in carrying out this program. First, one needs to determine for which $\widehat{G}$ such a sequence exists in $\ell^2$. Second, one needs to ensure the resulting weight $\rho = \sum a_n \rho_n$ is nonnegative. Third is the much more difficult question of which weighting functions can be physically realized by existing instrumentation.

This is closely related to the problem of determining the subspace of $L^2(0,t_i)$ spanned by the functions $\{ \exp(\lambda_n (t-t_i)) \}$, where $\{\lambda_n\}$ are the Sturm--Liouville eigenvalues from the adjoint problem (\ref{eqn:SL}). Such non-orthogonal series expansions arise in controllability problems for the heat equation (cf. \cite{FR71}), and there is extensive literature on their completeness properties (see e.g. \cite{PW33b,R68,S43}). %(see e.g. \cite{PW33b,R68,R77,RY83,S43}). 

We simplify the situation by just considering a single observation, at time $T > 0$, and asking: given $\widehat{G} \in L^2(0,T)$, does there exist a nonnegative weight $\rho = \sum a_n \rho_n$ in $L^2(0,h)$ such that
\[
	G_T(t) = k(0) r_T^{-1} \sum_{n=0}^{\infty} a_n e^{\lambda_n(t-T)}
\]
has a nontrivial projection onto $\widehat{G}$? In other words, we are requiring only that there is \emph{some} decrease in covariance along the given direction $\widehat{G}$, and allow for the possibility that the decrease may be greater elsewhere. We find that even this more modest demand cannot always be satisfied.
%, i.e. there is a function $G$ such that its orthogonal projection onto $G_T$ is zero for any possible choice of $\rho$.

\begin{theorem} \label{thm:unobservable}
There exists $G \in L^2(0,T)$ such that $\mathcal{C}_1 G = \mathcal{C}_0 G$ for any choice of $\rho$.
\end{theorem}
In other words, there are directions in which one can \emph{never} gain information by making an integrated observation as in (\ref{eqn:obs}).

%This should be contrasted with the initialization problem, in which the boundary conditions are known, and one is trying to infer the initial condition $q_0(z)$ from observations at later times. In that case, it is easy to see that the gain in information along the $n$th F

\begin{proof}
It follows from \eqref{eqn:cov} that $G$ will have the desired property if and only if
\begin{align*}
	\int_0^T G(t) e^{\lambda_n t} dt = 0
\end{align*}
for all $n \geq 0$.

However, the growth estimate (\ref{eqn:growth}) for the adjoint Sturm--Liouville eigenvalues $\{\lambda_n\}$ immediately gives
\begin{align}
	\sum_{n=0}^{\infty} \frac{1}{1+|\lambda_n|} < \infty,
	\label{eqn:sum}
\end{align}
in which case it is known \cite{PW33b,R68} that the set of functions $\{\exp(\lambda_n t)\}$ is not complete in $L^2(0,T)$. This completes the proof.
\end{proof}

There are two features of the above proof that are deserving of comment. First, the difficulty in targeting observations comes from the fact that the observations are spatial, integrated over the atmospheric column at a fixed time, while the quantity of interest---the flux--- is temporal, evaluated at a fixed point in space over a given interval of time. For instance, if we instead consider the initialization problem (where the boundary data is fixed, and the initial condition $q_0$ is the quantity of interest), it is not hard to see that an observation at time $t_i$ decreases the covariance in the direction spanned by
\begin{align}
	\sum_{n=0}^{\infty} a_n e^{-\lambda_n t_i} \rho_n(z)
\end{align}
in $L^2(0,h)$. Here there are no unobservable subspaces, because the set $\{\rho_n\}$ spans $L^2(0,h)$.

Second, we note the significance of the parabolicity of the transport equation. Since the equation is first-order in time, the functions $\{G_i\}$ defined in (\ref{eqn:basis}) involve sums of exponentials $\exp(\lambda_n t)$, with $\lambda_n \approx cn^2$ for large $n$, for which (\ref{eqn:sum}) is satisfied. In, on the other hand, $q$ solved a wave equation (\emph{e.g.} $q_{tt} = Lq$), the resulting posterior covariance formulas would involve terms of the form $\exp(\sqrt{\lambda_n} t)$, for which we obtain
\begin{align*}
	\sum_{n=0}^{\infty} \frac{1}{1+\sqrt{\lambda_n}} = \infty,
\end{align*}
hence the phenomenon of unobservability is not present in the hyperbolic case.

%%%%%%%%%%%%%%%%%%%%%%%%%%%%%%%%%%%%%%%%%%%%%%%%%%%%%%%%%%%%%%%%%%%%%%%%%%%%%%%%%%%%%%%%%%%%%%%%%%%%%%%%%%%%%%%%%%%%%%%%%%%%%%%%%%%%%%%%%%%%%%%%%%%%%%%%%%%%%%%%%%%%%%%%%%%%%%%%%%%%%%%%%%%%%%%%%%%%%%%%

\section*{Acknowledgments}
The author would like to thank Sean Crowell, David Kelly and Andrew Stuart for their helpful comments throughout the preparation of this work. This material is based upon work supported by the National Science Foundation under Grant Number DMS-1312906.%This research has been supported by the Office of Naval Research under the MURI grant N00014-11-1-0087.

\bibliographystyle{plain}
\bibliography{4Dvar}

\end{document}